\documentclass[12pt]{amsart}%
\usepackage{color}
\usepackage{amsmath}
\usepackage{amssymb}
\usepackage{amsfonts}
\usepackage[latin1]{inputenc}
\usepackage{graphicx}%
\providecommand{\U}[1]{\protect\rule{.1in}{.1in}}

\DeclareMathSymbol{\subsetneqq}{\mathbin}{AMSb}{36}

\theoremstyle{plain}
\numberwithin{equation}{section}
\newtheorem{theorem}{Theorem}[section]

\newtheorem{lemma}{Lemma}[section]

\newtheorem{remark}{Remark}[section]

\begin{document}
\begin{center}
Asymptotic for a second order evolution equation with convex potential and
vanishing damping term 

\end{center}

\begin{center}

Ramzi MAY

Department of Mathematics and Statistics

College of Sciences

King Faisal University

P.O. 400 Al Ahsaa 31982, Kingdom of Saudi Arabia

E-mail: rmay@kfu.edu.sa\bigskip
\end{center}

\noindent\textbf{Abstract:} In this short note, we recover by a different
method the new result due to Attouch, Chbani, Peyrouqet and Redont concerning
the weak convergence as $t\rightarrow+\infty$ of solutions $x(t)$ to the
second order differential equation%
\[
x^{\prime\prime}(t)+\frac{K}{t}x^{\prime}(t)+\nabla\Phi(x(t))=0,
\]
where $K>3$ and $\Phi$\ is a smooth convex function defined on an Hilbert
Space $\mathcal{H}.$ Moreover, we improve their result on the rate of
convergence of $\Phi(x(t))-\min\Phi.$\medskip

\noindent{\textbf{keywords}: dynamical systems, asymptotically small
dissipation, asymptotic behavior, energy function, convex function, convex
optimization.\medskip}

\section{Introduction and statement of the result}

Let $\mathcal{H}$ be a real Hilbert space with inner product and norm
respectively denoted by $\langle.,.\rangle$ and $\left\Vert .\right\Vert .$ In
a very recent work \cite{APR}, Attouch, Chbani, Peypouquet and Redont have
considered the following second order differential equation:
\begin{equation}
x^{\prime\prime}(t)+\gamma(t)x^{\prime}(t)+\nabla\Phi(x(t))=0,\label{eq1}%
\end{equation}
where $\gamma(t)=\frac{K}{t}$ with $K$ is a non negative constant and
$\Phi:\mathcal{H}\rightarrow\mathbb{R}$ is a convex continuously
differentiable function. By developing a method due to Su, Boyed, and Condes
\cite{SBC}, they proved the following result:

\begin{theorem}
[Attouch, Chbani, Peypouquet, and Redont]\label{th}Assume that $K>3$ and the
set $\arg\min\Phi\equiv\{x\in\mathcal{H}:~\Phi(x)\leq\Phi(y)~\forall
y\in\mathcal{H}\}$ is nonempty$.$ Let $x:[t_{0},+\infty\lbrack\rightarrow
\mathcal{H}$ be a solution to (\ref{eq1}). Then $x(t)$ converges weakly in
$\mathcal{H}$ as $t\rightarrow+\infty$ to some element of $\arg\min\Phi.$
Moreover the energy function%
\begin{equation}
W(t)\equiv\frac{1}{2}\left\Vert x^{\prime}(t)\right\Vert ^{2}+\Phi
(x(t))-\min\Phi\label{eq2}%
\end{equation}
satisfies $W(t)=O(t^{-2})$ as $t\rightarrow+\infty.$
\end{theorem}

In this note, we establish, by using a different method, a slightly improved
version of the previous theorem. Precisely, we prove the following result.

\begin{theorem}
\label{th1}Assume that $K>3$ and $\arg\min\Phi\neq\emptyset.$ Let
$x:[t_{0},+\infty\lbrack\rightarrow\mathcal{H}$ be a solution to (\ref{eq1}).
Then $x(t)$ converges weakly in $\mathcal{H}$ as $t\rightarrow+\infty$ to some
element of $\arg\min\Phi.$ Moreover $W(t)=\circ(t^{-2})$ as $t\rightarrow
+\infty.$
\end{theorem}

\begin{remark}
In \cite{M}, we studied the asymptotic behavior as $t\rightarrow+\infty$ of
solution to equation (\ref{eq1}) when the damping term $\gamma(t)$ behaves,
for $t$ large enough, like $\frac{K}{t^{\alpha}}$ with $K>0$ and $\alpha
\in\lbrack0,1[.$ We proved that if $\arg\min\Phi\neq\emptyset$ then every
solution to (\ref{eq1})\ converges weakly in $\mathcal{H}$ to some element of
$\arg\min\Phi.$ Hence, Theorem \ref{th} and Theorem \ref{th1}\ extend this
result to the limit case corresponding to $\alpha=1$.
\end{remark}

\section{Proof of Theorem \ref{th1}}

We will prove Theorem \ref{th1} in a more general setting. Indeed, we will
assume that the damping term $\gamma$ in Equation (\ref{eq1}) is a real
function defined on $[t_{0},+\infty\lbrack$ which belongs to the class
$W_{loc}^{1,1}([t_{0},+\infty\lbrack,\mathbb{R})$ and satisfies:
\begin{equation}
\text{There exists }K>3\text{ such that }\gamma(t)\geq\frac{K}{t}~\forall
t\geq t_{0}, \label{eq4}%
\end{equation}
and%
\begin{equation}
\int_{t_{0}}^{+\infty}\left[  (t\gamma(t))^{\prime}\right]  _{+}dt<+\infty,
\label{EQ}%
\end{equation}
where $\left[  (t\gamma(t))^{\prime}\right]  _{+}\equiv\max\{(t\gamma
(t))^{\prime},0\}$ is the positive part of $(t\gamma(t))^{\prime}.$

A typical examples of functions $\gamma$ satisfying (\ref{eq4}) and (\ref{EQ})
are $\gamma(t)=\frac{K}{a+t}$ with $a\in\mathbb{R}$ and $K>3.$

\begin{proof}
[Proof of Theorem \ref{th1}]We will use a modified version of a method
introduced by Cabot et Frankel in \cite{CF} and recently developed in \cite{M}.

\noindent Let $x^{\ast}\in\arg\min\Phi$ and define the function $h:[t_{0}%
,+\infty\lbrack\rightarrow\mathbb{R}^{+}$ by $h(t)=\frac{1}{2}\left\Vert
x(t)-x^{\ast}\right\Vert ^{2}.$ By differentiating, we have%
\begin{align*}
h^{\prime}(t)  &  =\langle x^{\prime}(t),x(t)-x^{\ast}\rangle,\\
h^{\prime\prime}(t)  &  =\left\Vert x^{\prime}(t)\right\Vert ^{2}+\langle
x^{^{\prime\prime}}(t),x(t)-x^{\ast}\rangle.
\end{align*}
Combining these last equalities and using Equation (\ref{eq1}), we get%
\begin{equation}
h^{\prime\prime}(t)+\gamma(t)h^{\prime}(t)=\left\Vert x^{\prime}(t)\right\Vert
^{2}+\langle\nabla\Phi(x(t)),x^{\ast}-x(t)\rangle. \label{equ}%
\end{equation}
Using now the convexity inequality%
\begin{equation}
\Phi(x^{\ast})\geq\Phi(x)+\langle\nabla\Phi(x),x^{\ast}-x\rangle, \label{M}%
\end{equation}
and the definition (\ref{eq2}) of the energy function $W,$ we obtain%
\begin{equation}
W(t)\leq\frac{3}{2}\left\Vert x^{\prime}(t)\right\Vert ^{2}-h^{\prime\prime
}(t)-\gamma(t)h^{\prime}(t). \label{eq5}%
\end{equation}
On the other hand, in view of (\ref{eq1}),%
\begin{align*}
W^{\prime}(t)  &  =\langle x^{\prime}(t),x(t)\rangle+\langle\nabla
\Phi(x(t)),x^{\prime}(t)\rangle\\
&  =-\gamma(t)\left\Vert x^{\prime}(t)\right\Vert ^{2}.
\end{align*}
Hence
\begin{equation}
(t^{2}W(t))^{\prime}=2tW(t)-t^{2}\gamma(t)\left\Vert x^{\prime}(t)\right\Vert
^{2}. \label{eq6}%
\end{equation}
Using now assumption (\ref{eq4}), we get%
\begin{align}
\frac{3}{2}t\left\Vert x^{\prime}\right\Vert ^{2}  &  \leq\frac{3}{2K}%
t^{2}\gamma(t)\left\Vert x^{\prime}(t)\right\Vert ^{2}\nonumber\\
&  =\frac{3}{K}tW(t)-\frac{3}{2K}\left(  t^{2}W(t)\right)  ^{\prime}.
\label{eq7}%
\end{align}
Multiplying (\ref{eq5}) by $t$ and using Inequality (\ref{eq7}), we obtain%
\[
(1-\frac{3}{K})tW(t)+\frac{3}{2K}\left(  t^{2}W(t)\right)  ^{\prime}%
\leq-th^{\prime\prime}(t)-t\gamma(t)h^{\prime}(t).
\]
Integrating this last inequality on $[t_{0},t],$ we get after simplification%
\begin{align}
(1-\frac{3}{K})\int_{t_{0}}^{t}sW(s)ds+\frac{3}{2K}\left(  t^{2}W(t)\right)
&  \leq C_{0}-th^{\prime}(t)+(1-t\gamma(t))h(t)\nonumber\\
&  +\int_{t_{0}}^{t}(s\gamma(s))^{\prime}h(s)ds, \label{Ra}%
\end{align}
where $C_{0}=\frac{3}{2K}\left(  t_{0}^{2}W(t_{0})\right)  +t_{0}h^{\prime
}(t_{0})-h(t_{0}).$

\noindent Let $\varepsilon>0$ such that $K>3+3\varepsilon.$ By using
(\ref{eq4}), we obtain from the inequality (\ref{Ra})%
\begin{align*}
(1-\frac{3}{K})\int_{t_{0}}^{t}sW(s)ds+\frac{3}{2K}\left(  t^{2}W(t)\right)
+\varepsilon h(t)  &  \leq C_{0}-th^{\prime}(t)-(K-1-\varepsilon)h(t)\\
&  +\int_{t_{0}}^{t}\left[  (s\gamma(s))^{\prime}\right]  _{+}h(s)ds.
\end{align*}
Using now the fact that%
\begin{align*}
t\left\vert h^{\prime}(t)\right\vert  &  \leq t\left\Vert x^{\prime
}(t)\right\Vert \left\Vert x(t)-x^{\ast}\right\Vert \\
&  \leq2\sqrt{t^{2}W(t)}\sqrt{h(t)},
\end{align*}
and applying the elementary inequality
\[
\forall a>0\forall b,x\in\mathbb{R},~-ax^{2}+bx\leq\frac{b^{2}}{4a}%
\]
with $x=\sqrt{h(t)},$ we get%
\begin{equation}
A\int_{t_{0}}^{t}sW(s)ds+Bt^{2}W(t)+\varepsilon h(t)\leq C_{0}+\int_{t_{0}%
}^{t}\left[  (s\gamma(s))^{\prime}\right]  _{+}h(s)ds \label{Najla}%
\end{equation}
where $A=1-\frac{3}{K}$ and $B=\frac{3}{2K}-\frac{1}{K-1-\varepsilon}.$

\noindent Since $K>3+3\varepsilon,$ the constants $A$ and $B$ are positive,
then%
\[
\varepsilon h(t)\leq C_{0}+\int_{t_{0}}^{t}\left[  (s\gamma(s))^{\prime
}\right]  _{+}h(s)ds.
\]
Hence, by using Gronwall's inequality and the assumption (\ref{EQ}), we deduce
that the function $h$ is bounded, more precisely, we get
\[
\sup_{t\geq t_{0}}h(t)\leq\frac{C_{0}}{\varepsilon}\exp(\frac{1}{\varepsilon
}\int_{t_{0}}^{+\infty}\left[  (s\gamma(s))^{\prime}\right]  _{+}ds).
\]
Therefore, we infer from (\ref{Najla}) that%

\begin{align}
\sup_{t\geq t_{0}}t^{2}W(t)  &  <+\infty,\label{eq9}\\
\int_{t_{0}}^{+\infty}sW(s)ds  &  <+\infty. \label{eq10}%
\end{align}
Combining (\ref{eq6}) and (\ref{eq10}) yields that the positive part $\left[
(t^{2}W(t))^{\prime}\right]  _{+}$ of $(t^{2}W(t)^{\prime}$ belongs to
$L^{1}([t_{0},+\infty\lbrack,\mathbb{R}),$ hence $m:=\lim_{t\rightarrow
+\infty}t^{2}W(t)$ exists. This limit $m$ must be equal to $0,$ since
otherwise $tW(t)\simeq\frac{m}{t}$ as $t\rightarrow+\infty$ which contradicts
(\ref{eq10}). It remains to prove the weak convergence of $x(t)$ as
$t\rightarrow+\infty.$ Let us notice that (\ref{eq9}) implies that
$\Phi(x(t))\rightarrow\min\Phi$ as $t\rightarrow+\infty.$ Hence by using the
weak lower semi-continuity of the function $\Phi,$ we deduce that if
$x(t_{n})\rightharpoonup\bar{x}$ weakly in $\mathcal{H}$ with $t_{n}%
\rightarrow+\infty$ then $\Phi(\bar{x})\leq\min\Phi$ which is equivalent to
$\bar{x}\in\arg\min\Phi.$ On the other hand, from the convex inequality
(\ref{M}) we deduce that $\langle\nabla\Phi(x),x^{\ast}-x\rangle\leq0$ for
every $x\in\mathcal{H}.$ Then Equation (\ref{equ}) implies%
\[
h^{\prime\prime}(t)+\gamma(t)h^{\prime}(t)\leq\left\Vert x^{\prime
}(t)\right\Vert ^{2}.
\]
Multiply this last equation by $e^{\Gamma(t,t_{0})},$ where $\Gamma
(t,s)=\int_{s}^{t}\gamma(\tau)dt,$ and integrate between $t_{0}$ and $t$, we
obtain%
\begin{equation}
h^{\prime}(t)\leq e^{-\Gamma(t,t_{0})}h^{\prime}(t_{0})+\int_{t_{0}}%
^{t}e^{-\Gamma(t,\tau)}\left\Vert x^{\prime}(\tau)\right\Vert ^{2}d\tau.
\label{eq11}%
\end{equation}
In view of the assumption (\ref{eq4}), a simple calculation gives%
\[
\forall s\geq t_{0},~\int_{s}^{+\infty}e^{-\Gamma(t,s)}dt\leq\frac{s}{K-1}.
\]
Hence by using (\ref{eq11}) and Fubini Theorem, we get%
\[
\int_{t_{0}}^{+\infty}[h^{\prime}(t)]_{+}dt\leq\frac{t_{0}\left\vert
h^{\prime}(t_{0})\right\vert }{K-1}+\frac{1}{K-1}\int_{t_{0}}^{+\infty}%
\tau\left\Vert x^{\prime}(\tau)\right\Vert ^{2}d\tau.
\]
Thanks to (\ref{eq10}), the right hand side of the last inequality is finite,
thus $\int_{t_{0}}^{+\infty}[h^{\prime}(t)]_{+}dt<+\infty$ which implies that
$\lim_{t\rightarrow+\infty}h(t)$ exists. Hence, for every $x^{\ast}\in\arg
\min\Phi,$ the limit of $\left\Vert x(t)-x^{\ast}\right\Vert $ as
$t\rightarrow+\infty$ exists. Therefore, Opial's lemma \cite{Op}, which we
recall below, guaranties the required weak convergence of $x(t)$ in
$\mathcal{H}$ to some element of $\arg\min\Phi.$
\end{proof}

\begin{lemma}
[Opial's lemma]Let $x:[t_{0},+\infty\lbrack\rightarrow\mathcal{H}.$ Assume
that there exists a nonempty subset $S$ of $\mathcal{H}$ such that:

\begin{enumerate}
\item[i)] If $t_{n}\rightarrow+\infty$ and $x(t_{n})\rightharpoonup x$ weakly
in $\mathcal{H}$ , then $x\in S.$

\item[ii)] For every $z\in S,$ $\lim_{t\rightarrow+\infty}\left\Vert
x(t)-z\right\Vert $ exists.
\end{enumerate}

\noindent Then there exists $z_{\infty}\in S$ such that $x(t)\rightharpoonup
z_{\infty}$ weakly in $\mathcal{H}$ as $t\rightarrow+\infty.$
\end{lemma}

\noindent\textbf{Conclusion:} In this paper, we have proved that if the
damping term $\gamma(t)$ behaves at infinity like $\frac{K}{t}$ with $K>3,$
then every solution $x(t)$ of the equation (\ref{eq1}) converges weakly as
$t\rightarrow+\infty$ to a minimizer of $\Phi$ and the energy function $W(t)$
is $\circ(t^{-2})$. However, two important questions
remain open. The first one is on the behavior of the solution
$x(t)$  in the limit case $K=3$ and the second one is  about the effect of
the constant $K$ on the convergence rate of the associated energy function $W(t)$.\smallskip

\noindent \textbf{Acknowledgement:} The author wish to thank Prof. Adel Trabelsi for his comments which were very
useful to improve the representation of the paper.


\begin{thebibliography}{99}

\bibitem {APR}Attouch H, Chbani Z, Peypouquet J, Redont P. Fast convergence of inertial dynamics and algorithms 
with asymptotic vanishing viscosity. Math Program Ser B 2016: 1-53.

\bibitem {CF}Cabot A, Frankel P. Asymptotics for some semilinear hyperbolic
equations with non-autonomous damping. J Differ Equations 2012; 252: 294-322.

\bibitem {M} May R. Long time behavior for a semilinear hyperbolic equation
with asymptotically vanishing damping term and convex potential. J Math
Anal Appl 2015; 430: 410-416.

\bibitem {SBC}Su W, Boyd S, Candes E J.  A differential equation for
modeling Nestrov's accelerated gradient method: Theory and Insights. Neural
Information Proceeding Systems (NIPS). 2014.

\bibitem {Op}Opial Z. Weak convergence of the sequence of successive
aproximation for nonexpansive mapping. Bull Amer Math Soc 1967; 73: 591-597.
\end{thebibliography}
\end{document}